\documentclass[10pt]{amsart}

\usepackage{amssymb, amsmath, amsfonts, stmaryrd}
\usepackage[dvips]{graphicx}
\usepackage[dvips]{color}
\usepackage[noend]{algorithmic}
\usepackage[ruled]{algorithm}
\usepackage[all]{xy}
\usepackage{a4wide}
\newcommand{\todo}[1]{%
\marginpar{\begin{flushleft}\textbf{#1}\end{flushleft}}%
}
\renewcommand{\todo}[1]{}

\newcommand{\lp}{\left(\!\left(}
\newcommand{\rp}{\right)\!\right)}
\renewcommand{\ll}{\llbracket}
\newcommand{\rr}{\rrbracket}
\newcommand{\N}{\mathbb{N}}
\newcommand{\Z}{\mathbb{Z}}
\newcommand{\Q}{\mathbb{Q}}
\newcommand{\AO}{\operatorname{O}}

\newcommand{\E}{\mathbb{E}}
\newcommand{\F}{\mathbb{F}}
\renewcommand{\P}{\mathbb{P}}
\renewcommand{\d}{\operatorname{d}}
\newcommand{\OO}{\mathcal{O}}
\newcommand{\mm}{\mathfrak{m}}
\renewcommand{\SS}{\mathcal{S}}

\newcommand{\into}{\hookrightarrow}
\newcommand{\ord}{\operatorname{ord}}
\newcommand{\spec}{\operatorname{Spec}}

\newcommand{\cent}{\operatorname{center}}
\newcommand{\supp}{\operatorname{supp}}
\newcommand{\totfrac}{\mathcal{Q}}
\newcommand{\ratto}{\dashrightarrow}

\newtheorem{thm}{Theorem}[section]
\newtheorem{lem}[thm]{Lemma}
\newtheorem{cor}[thm]{Corollary}
\newtheorem{rem}[thm]{Remark}
\newtheorem{dfn}[thm]{Definition}

\numberwithin{equation}{section}

\begin{document}

\title{Adjoint Computation for Hypersurfaces Using Formal Desingularizations}

\author{Tobias Beck and Josef Schicho}
\address{Symbolic Computation Group, Johann Radon Institute for Computational and Applied Mathematics, Austrian Academy of Sciences, Altenberger Stra{\ss}e 69, A-4020 Linz, Austria}
\email{Tobias.Beck@ricam.oeaw.ac.at}
\email{Josef.Schicho@ricam.oeaw.ac.at}

\thanks{The authors were supported by the FWF (Austrian Science Fund) in the frame of the research project SFB F1303 (part of the  Special Research Program ``Numerical and Symbolic Scientific Computing'').}
\date{\today}

\begin{abstract}
We show how to use formal desingularizations (defined earlier by the first author) in order to compute the global sections (also called adjoints) of twisted pluricanonical sheaves. These sections define maps that play an important role in the birational classification of schemes.
\end{abstract}

\maketitle

\tableofcontents

\section{Introduction}

In \cite{TB:2007} we have introduced formal desingularizations of schemes and shown how to compute them efficiently for surfaces $X \subset \P^3$. There we have already mentioned as an application the computation of invertible sheaves on $Y$ for a certain desingularization $\pi: Y \to X$. In this report we want to make this precise using the powers of the canonical sheaf on $Y$. More precisely we want to determine the coherent sheaf $\pi_*(\omega_Y^{\otimes m})$ which can be given by its associated graded module $$\Gamma_*(\pi_*(\omega_Y^{\otimes m})) := \bigoplus_{n \in \Z} \Gamma(X, \pi_*(\omega_Y^{\otimes m}) \otimes_{\OO_X} \OO_X(n)),$$ see \cite[p.\ 118]{RH:1977}. The sheaves $\pi_*(\omega_Y^{\otimes m}) \otimes_{\OO_X} \OO_X(n)$ are called \emph{twisted pluricanonical}. Our Main Theorem~\ref{thm.FormalCrit} gives a criterion for computing its homogeneous components. The theory is independent of the dimension and given for hypersurfaces of some projective space.

The importance of the components of $\Gamma_*(\pi_*(\omega_Y^{\otimes m}))$ stems from the birational classification of schemes. For example, they provide an effective way to check Castelnuovo's Criterion and perform the Enriques-Manin Reduction of rational surfaces \cite{JS:1998b}.

This report is structured as follows: In Section~\ref{sec.Definition} we start by recalling the definition of formal desingularizations for the convenience of the reader. In Section~\ref{sec.Adjoints} we define the sheaf of $m$-adjoints on $X$ by a property involving formal prime divisors and show that it is isomorphic to $\pi_*(\omega_Y^{\otimes m})$. In particular it is independent of a special $Y$. In Section~\ref{sec.Computing} we find a super-sheaf of $\pi_*(\omega_Y^{\otimes m})$ and show that the defining property of the latter is easily checked using a formal desingularization of $X$. This immediately yields Algorithm~\ref{alg.Adjoints} given in Section~\ref{sec.Algorithm}. We close with an example.

Before we proceed we recall and fix some notions. Let $\E$ be a field of characteristic zero and $X$ and $Y$ \emph{integral $\E$-schemes}. All (rational) maps are relative over $\spec \E$. By $\E(X)$ and $\E(Y)$ we denote the respective function fields. A \emph{rational map} $\pi: Y \ratto X$ is given by a tuple $(V,\pi)$ s.t.\ $V \subseteq Y$ is open and $\pi: V \to X$ is a regular morphism. Note that we do not restrict to schemes of finite type here. In particular all regular morphisms are rational maps. Two tuples $(V_1,\pi_{1})$ and $(V_2,\pi_{2})$ are \emph{equivalent}, or define the same rational map, if $\pi_{1}|_{V_1 \cap V_2} = \pi_{2}|_{V_1 \cap V_2}$.

Assume that two maps send the generic point of $Y$ to $p \in X$ (its image is always defined for rational maps). Then $(V_1,\pi_{1})$ and $(V_2,\pi_{2})$ are equivalent iff the induced inclusions of fields $\OO_{X,p}/\mm_{X,p} \into \E(Y)$ are the same (where $\mm_{X,p} \subset \OO_{X,p}$ is the maximal ideal). In particular if $\pi$ is \emph{dense}, {\it i.e.}, $p$ is the generic point of $X$, we get an inclusion $\E(X) \into \E(Y)$ determining $\pi$.

Note, however, that not all such field inclusions yield rational maps under our assumption since we have not yet restricted to schemes of finite type over $\E$. {\it E.g.}, let $X := \spec \E[x]$, $Y := \spec \E[x]_{\langle x \rangle}$ and $\pi: Y \to X$ be the morphism induced by localization. Then $\pi$ induces an isomorphism of function fields $\E(X) \cong \E(Y)$. Nevertheless $\pi$ has no rational inverse. A rational map with inverse is called \emph{birational} (or also a \emph{birational transformation}).

Further it is easy to see that dense rational maps may be composed. A rational map has a \emph{domain of definition}, which is the maximal open set on which it can be defined (equivalently, the union of all such open sets).

\section{Definition of Formal Desingularizations}\label{sec.Definition}

From now on $X$ and $Y$ will denote \emph{separated, integral schemes of finite type over $\E$} and they will have the same dimension $l$. Let $(A,\mathfrak{m})$ be a valuation ring of $\E(X)$ over $\E$ (where $\mathfrak{m}$ is the maximal ideal). If $A$ is discrete of rank $1$ and the transcendence degree of $A/\mathfrak{m}$ over $\E$ is $n-1$ then it is called a \emph{divisorial valuation ring of $\E(X)$ over $\E$} or a \emph{prime divisor of $\E(X)$} (see, {\it e.g.}, \cite[Def.~2.6]{MS:1990}). It is an \emph{essentially finite}, regular, local $\E$-algebra of Krull-dimension $1$ ({\it i.e.}, the localization of a finitely generated $\E$-algebra at a prime ideal, see \cite[Thm.~VI.14.31]{OZ_PS:1960}).

Let $(A,\mathfrak{m})$ be a divisorial valuation ring of $\E(X)$ over $\E$. By \cite[Lem.~II.4.4.]{RH:1977} the inclusion $A \subset \E(X)$ defines a unique morphism $\spec \totfrac(A) \to X$ and therefore a rational map $\spec A \ratto X$ sending generic point to generic point. Composing this with the morphism obtained by the $\mathfrak{m}$-adic completion $A \to \widehat{A}$ we get a rational map $\spec \widehat{A} \ratto X$ in a natural way.

\begin{dfn}[Formal Prime Divisor]\label{dfn.FormalDiv}
Let $(A,\mathfrak{m})$ be a divisorial valuation ring of $\E(X)$ over $\E$. Assume that the rational map $\spec \widehat{A} \ratto X$ (as above) is actually a \emph{morphism} $\varphi: \spec \widehat{A} \to X$ ({\it i.e.}, defined also at the closed point). Then $\varphi$ is a representative for a class of schemes up to $X$-isomorphism. This class (and, by abuse of notation, any representative) will be called a \emph{formal prime divisor} on $X$.
\end{dfn}

Hence we may compose a representative $\varphi$ with an isomorphism $\spec B \to \spec \widehat{A}$ to get another representative for the same formal prime divisor. By the \emph{Cohen Structure Theorem} (see, {\it e.g.}, \cite[Thm.~7.7]{DE:1995} with $I=0$) we know that $\widehat{A} \cong \F_\varphi\ll t\rr $ with $\F_\varphi := A/\mathfrak{m} \cong \widehat{A}/\mathfrak{m}\widehat{A}$. Therefore we will sometimes assume that $\varphi$ is of the form $\spec \F_\varphi\ll t\rr \to X$. 

Formal prime divisors provide an algorithmic way for dealing with certain valuations; A formal prime divisors yields an inclusion of function fields $\E(X) \into \F_\varphi\lp t\rp $. Vice versa, by what was said above, $\varphi$ is determined by this inclusion. So it is this piece of information one has to compute (see Remark~\ref{rem.FormalDesingComp} below). Composing this inclusion with the order function $\ord_t: \F_\varphi\lp t \rp \to \Z$ we get the corresponding divisorial valuation (see Definition~\ref{dfn.AssocVals} below).

We want to single out a special class of formal prime divisors.

\begin{dfn}[Realized Formal Prime Divisors]\label{dfn.RealizedDiv}
Let $p \in X$ be a regular point of codimension $1$. The formal prime divisor $$\spec \widehat{\OO_{X,p}} \to X$$ (given by composing the canonic morphism $\spec \OO_{X,p} \to X$ with the morphism induced by the completion $\OO_{X,p} \to \widehat{\OO_{X,p}}$) is called \emph{realized}.
\end{dfn}

If $X$ is \emph{normal} then all generic points of closed subsets of codimension $1$ are necessarily regular \cite[Thm~II.8.22A]{RH:1977}. Therefore there is a one-one correspondence of realized formal prime divisors and prime Weil divisors. Another important fact is that we can match the formal prime divisors of birationally equivalent schemes under certain conditions.

\begin{lem}[Pullback along Proper Morphisms]\label{lem.DivPullback}
Let $\pi: Y \to X$ be a proper, birational morphism. A formal prime divisor $\varphi: \spec \F_\varphi\ll t\rr \to X$ lifts to a unique formal prime divisor $\pi^* \varphi: \spec \F_\varphi\ll t\rr \to Y$ s.t.\ $\pi \circ (\pi^* \varphi) = \varphi$. Vice versa, a formal prime divisor on $Y$ extends to a unique formal prime divisor on $X$, hence $\pi^*$ is a bijection.
\end{lem}
\begin{proof}
See \cite[Cor.\ 2.4]{TB:2007}.
\end{proof}

We will apply the operator $\pi^*$ also to sets of formal prime divisors.

\begin{dfn}[Center and Support]\label{dfn.CenterSupp}
Let $\varphi: \spec \F_\varphi\ll t\rr \to X$ be a formal prime divisor. We define its \emph{center}, in symbols $\cent(\varphi) \in X$, to be the image of the closed point. Further the \emph{support} of a \emph{finite} set of formal prime divisors $\SS$ is defined as $\supp(\SS) := \overline{\{\cent(\varphi) \mid \varphi \in \SS \}}$, {\it i.e.}, the closure of the set of all centers.
\end{dfn}

Now we are in the situation to define formal desingularizations.

\begin{dfn}[Formal Description of a Desingularization]\label{dfn.FormalDesc}
Let $\pi: Y \to X$ be a \emph{desingularization}, {\it i.e.}, $\pi$ is proper, birational and $Y$ is regular. Let $\SS$ be a \emph{finite} set of formal prime divisors on $X$. We say that $\SS$ is a \emph{formal description} of $\pi$ iff
\begin{enumerate}
\item all divisors in $\pi^* \SS$ are realized,
\item $\pi^{-1}(\supp(\SS)) = \supp(\pi^* \SS)$ and
\item the restricted morphism $Y \setminus \supp(\pi^* \SS) \to X \setminus \supp(\SS)$ is an isomorphism.
\end{enumerate}
\end{dfn}

The set $\SS$ itself consists of formal prime divisors on $X$ and makes no reference to the morphism $\pi$. By another definition we can \emph{avoid} mentioning any \emph{explicit} $\pi$.

\begin{dfn}[Formal Desingularization]\label{dfn.FormalDesing}
Let $\SS$ be a finite set of formal prime divisors on $X$. Then $\SS$ is called a \emph{formal desingularization} of $X$ iff there \emph{exists some} desingularization $\pi$ s.t.\ $\SS$ is a formal description of it.
\end{dfn}

In \cite[Thm.\ 2.9]{TB:2007} it is shown that such $\pi$ is unique up to $X$-isomorphism if $X$ is a surface. In this case we also have an efficient algorithm to compute a set $\SS$. Informally speaking $\SS$ makes it possible to work with invertible sheaves on $Y$, although $Y$ is not constructed explicitely.

\begin{rem}[Formal Desingularizations in Higher Dimensions]\label{rem.FormalDesingComp}
This paper deals with projective hypersurfaces of arbitrary dimension $l$. Also in this case, formal desingularizations exist but are not so easy to compute. We want to indicate how a formal desingularization could be obtained by an ad hoc method (modulo a means to represent algebraic power series).

First compute a desingularization $\pi: Y \to X$, for example, by {\sc Villamayor}'s algorithm \cite{SE_OV:2000}. Let $Z \subset X$ be the singular locus of $X$. The algorithm will produce $\pi$ s.t.\ $\pi^{-1}(Z) \subset Y$ is a normal crossing divisor and $\pi$ restricts to an isomorphism on $Y \setminus \pi^{-1}(Z)$. Let $\{p_1,\dots,p_r\} \subset Y$ be the \emph{finitely many} generic points of the irreducible components of $\pi^{-1}(Z)$.

Next we have to compute isomorphisms $\widehat{\OO_{Y,p_i}} \to \F_i\ll t\rr $ where $\F_i := \OO_{Y,p_i} / \mm_{Y,p_i}$. Therefore let $U_i \cong \spec \E[x_{i, 1},\dots,x_{i, m_i}] / \langle f_{i,1}, \dots, f_{i,n_i} \rangle$ be an affine neighborhood of $p_i$. Constructing the isomorphism involves finding certain ``minimal'' algebraic power series $X_{i,1}, \dots, X_{i,m_i} \in \F_i\ll t\rr $ (compare \cite[Cor.\ A.2]{TB:2007}) that simultaneously solve $f_{i,1}, \dots, f_{i,n_i}$, essentially, computing a Taylor expansion. These power series together with $\pi$ and the inclusions $U_i \into Y$ can be used to represent a formal prime divisor $\varphi_i$ via, for example, the induced embedding of function fields $\E(X) \into \F_i\lp t\rp $. Set $\SS := \{ \varphi_1, \ldots, \varphi_r \}$.

This approach of course is not practical. It suffers from the huge computational overhead that the general resolution machinery involves. Also we are not very flexible with regard to the representation of the blown up schemes, thus annihilating the benefits of formal descriptions.
\end{rem}


\section{Adjoint Differential Forms}\label{sec.Adjoints}

We write $$\Omega_{X,rat}^m := (\Omega_{\E(X) \mid \E}^{\wedge l})^{\otimes m}$$ for the $m$-fold tensor power of the rational differential $l$-forms (which is an $1$-dimensional $\E(X)$-vector space by \cite[Thm.~16.14]{DE:1995} and \cite[Prop.~XIX.1.1 \& Cor.~XVI.2.4]{SL:2002} and can as well be considered a constant sheaf of $\OO_X$-modules).

Let $\varphi: \widehat{A} \to X$ be a formal prime divisor. We define $$\widetilde{\Omega}_{\widehat{A} \mid \E}^m := \widehat{A} \otimes_{A} (\Omega_{A \mid \E}^{\wedge l})^{\otimes m} \text{ (canonically included in } \widetilde{\Omega}_{\totfrac(\widehat{A}) \mid \E}^m := \totfrac(\widehat{A}) \otimes_{A} (\Omega_{A \mid \E}^{\wedge l})^{\otimes m}),$$ which is the $m$-fold tensor power of the \emph{universally finite module of $l$-differentials of $\widehat{A}$} (see \cite[Cor.~12.5]{EK:1986}). Note that we have a derivation $\d: \totfrac(\widehat{A}) \to \totfrac(\widehat{A}) \otimes_{A} \Omega_{A \mid \E}$.

This module is independent of the choice of a representative of $\varphi$ up to $X$-isomorphism because $\widetilde{\Omega}_{\widehat{A} \mid \E}^m$ can be defined by universal properties. So we can again substitute $\widehat{A}$ by $\F_{\varphi}\ll t\rr $. If $s_{\varphi,1},\ldots,s_{\varphi,l-1} \in \F_\varphi$ is a transcendence basis over $\E$ then $\{\d s_{\varphi,1}, \ldots, \d s_{\varphi,l-1}, \d t\}$ is a free basis of $\Omega_{\F_\varphi[t]_{\langle t\rangle} \mid \E}$ by \cite[Thm.~16.14]{DE:1995} and \cite[Thm.~25.1]{HM:1989}. Therefore also $\widetilde{\Omega}_{\F_\varphi\ll t\rr  \mid \E}^m$ is free of rank $1$ generated by $(\d s_{\varphi,1} \wedge \ldots \wedge \d s_{\varphi,l-1} \wedge \d t)^{\otimes m}$. (Note that the module of K{\"a}hler differentials $\Omega_{\F_\varphi\ll t\rr  \mid \E}$ would not be finitely generated.)

We have seen above that a formal prime divisor $\varphi: \spec \widehat{A} \to X$ induces an embedding $\varphi^\# : \E(X) \to \F_{\varphi}\lp t\rp$ of function fields. This again induces embeddings $\varphi^\# : \Omega_{X,rat}^m \to \widetilde{\Omega}_{\F_{\varphi}\lp t\rp \mid \E}^m$ in the obvious way ({\it i.e.}, $\d f \mapsto \d \varphi^\#(f)$).

\begin{dfn}[Regularity of Forms at Formal Prime Divisors]\label{dfn.RegAtPrimeDiv}
Let $\varphi: \spec \F_{\varphi}\ll t\rr \to X$ be a formal prime divisor. We say that $\eta \in \Omega_{X,rat}^m$ is \emph{regular at $\varphi$} iff $\varphi^\#(\eta) \in \widetilde{\Omega}_{\F_{\varphi}\ll t\rr \mid \E}^m$.
\end{dfn}

This manner of speaking is justified by Lemma~\ref{lem.RegRealized} below.

\begin{dfn}[Sheaf of Adjoint Forms]\label{dfn.AdjForm}
The map $$U \mapsto \{ \eta \in \Omega_{X,rat}^m \mid \eta \text{ is regular at all formal prime divisors on $X$ centered in $U$} \}$$ for all open subsets $U \subseteq X$ defines a subsheaf which we call the \emph{sheaf of $m$-adjoint forms (or just $m$-adjoints)}, in symbols $\Omega_{X,adj}^m$.
\end{dfn}

By what we have worked out so far we immediately find a nice property of adjoint forms.

\begin{cor}[Covariance of Adjoint Forms]\label{cor.AdjCov}
The sheaves of adjoints are covariants under proper, birational morphisms: If $\pi: Y \to X$ is a proper, birational morphism then $\pi_*(\Omega_{Y,adj}^m) = \Omega_{X,adj}^m$ as subsheaves of $\Omega_{X,rat}^m$.
\end{cor}
\begin{proof}
Since $\pi$ is birational we get a vector space isomorphism $\pi^\# : \Omega_{X,rat}^m \to \Omega_{Y,rat}^m$. With this identification $\pi_*(\Omega_{Y,adj}^m)$ becomes a subsheaf of $\Omega_{X,rat}^m$. The rest follows from the above definitions and Lemma~\ref{lem.DivPullback}.
\end{proof}

By $\Omega_{X,reg}^m := \bigotimes (\bigwedge \OO_X \d \OO_X) \subset \Omega_{X,rat}^m$ we denote the subsheaf of regular forms, {\it i.e.}, all forms locally expressible by sections of $\OO_{X}$. More precisely, if $\Omega_{X \mid \E}$ is the usual sheaf of K{\"a}hler differentials then we mean its image under the natural map $\iota: (\Omega_{X \mid \E}^{\wedge l})^{\otimes m} \to \Omega_{X,rat}^m$. Note that we do not have $(\Omega_{X \mid \E}^{\wedge l})^{\otimes m} \cong \Omega_{X,reg}^m$ in general; At singular points the K{\"a}hler differentials need not be torsion free (and neither their exterior and tensor products) whereas $\Omega_{X,reg}^m \subset \Omega_{X,rat}^m$ always is. At a regular point $p$, on the contrary, $((\Omega_{X \mid E}^{\wedge l})^{\otimes m})_p$ is free of rank $1$ (see \cite[Thm.~II.8.15]{RH:1977}).\todo{Better citation for non-closed fields?}
Therefore $\iota$ is locally an inclusion at $p$. In the next three lemmas we want to explore in detail the relation between the concepts of adjointness and regularity (at points or formal prime divisors) for forms in $\Omega_{X,rat}^m$.

\begin{lem}[Adjoint Forms and Regular Forms]\label{lem.AdjReg1}
Let $p \in X$ be a point. Then $(\Omega_{X,reg}^m)_p \subseteq (\Omega_{X,adj}^m)_p$ as subsheaves of $\Omega_{X,rat}^m$.
\end{lem}
\begin{proof}
Assume $\eta \in (\Omega_{X,reg}^m)_p$ or equivalently $\eta \in \bigotimes (\bigwedge \Gamma(U,\OO_{X}) \d \Gamma(U,\OO_{X}))$ for some open neighborhood $U$ of $p$. We have $\eta \in \Gamma(U,\Omega_{X,adj}^m)$ iff $\varphi^\#(\eta) \in \widetilde{\Omega}_{\F_\varphi\ll t\rr  \mid \E}^m = \bigotimes (\bigwedge \F_\varphi\ll t\rr  \d \F_\varphi\ll t\rr )$ for all formal prime divisors $\varphi: \spec \F_\varphi\ll t\rr  \to X$ centered in $U$. But for such $\varphi$ we must have $\varphi^\#(\Gamma(U,\OO_{X})) \subseteq \F_\varphi\ll t\rr $ and hence this condition is trivially fulfilled. A fortiori $\eta \in (\Omega_{X,adj}^m)_p$.
\end{proof}

\begin{lem}[Regularity and Realized Formal Prime Divisors]\label{lem.RegRealized}
Let $\varphi$ be a realized formal prime divisor on $X$, $p := \cent(\varphi)$ and $\eta \in \Omega_{X,rat}^m$. Then $\eta \in (\Omega_{X,reg}^m)_p$ iff $\eta$ is regular at $\varphi$.
\end{lem}
\begin{proof}
By Lemma~\ref{lem.AdjReg1} it remains to show that regularity at $\varphi$ implies $\eta \in (\Omega_{X,reg}^m)_p$. Since $\varphi$ is realized it is of the form $\spec \widehat{\OO_{X,p}} \to X$. Let $\gamma \in (\Omega_{\OO_{X,p} \mid \E}^{\wedge l})^{\otimes m}$ be a generator. Write $\eta = a/b\,\gamma$ with $a,b \in \OO_{X,p}$. Then $\varphi^\#(\eta) \in \widetilde{\Omega}_{\widehat{\OO_{X,p}} \mid \E}^m$ implies $b \vert a$ in $\widehat{\OO_{X,p}}$, in other words $a \in b\,\widehat{\OO_{X,p}}$. But $b\,\widehat{\OO_{X,p}} \cap \OO_{X,p} = b\,\OO_{X,p}$ by \cite[Thm.~7.5(ii)]{HM:1989} (because completion is \emph{faithfully flat} \cite[Thm.~7.2 \& Thm.~8.8]{HM:1989}). Therefore $b \vert a$ in $\OO_{X,p}$ and $\eta \in (\Omega_{X,reg}^m)_p$.
\end{proof}

\begin{lem}[Adjoint Forms at Regular Points]\label{lem.AdjReg2}
Let $p \in X$ be a \emph{regular} point. Then $(\Omega_{X,reg}^m)_p = (\Omega_{X,adj}^m)_p$ as subsheaves of $\Omega_{X,rat}^m$.
\end{lem}
\begin{proof}
By Lemma~\ref{lem.AdjReg1} it remains to show $(\Omega_{X,adj}^m)_p \subseteq (\Omega_{X,reg}^m)_p$. Assume indirectly that $\eta \in (\Omega_{X,adj}^m)_p$ but $\eta \not\in (\Omega_{X,reg}^m)_p$. Regularity is an open property and $\Omega_{X,reg}^m$ is free of rank $1$ in a neighborhood of $p$. Since $p$ is regular (in particular normal) we must have $\eta \not\in (\Omega_{X,reg}^m)_q$ for some point $q$ s.t.\ $q$ is regular, of codimension $1$ and $p \in \overline{q}$ (see \cite[Cor.~11.4]{DE:1995}). Consider the realized formal prime divisor $\varphi: \spec \widehat{\OO_{X,q}} \to X$. Lemma~\ref{lem.RegRealized} above implies that $\eta$ is not regular at $\varphi$. But $q = \cent(\varphi)$ is contained in any open neighborhood of $p$ contradicting $\eta \in (\Omega_{X,adj}^m)_p$.
\end{proof}

Now assume that $Y$ is \emph{regular}. In this situation one has $(\Omega_{Y \mid \E}^{\wedge l})^{\otimes m} \cong \Omega_{Y,reg}^m$. In terms of the canonical sheaf this means $\omega_{Y}^{\otimes m} \cong \Omega_{Y,reg}^m = \Omega_{Y,adj}^m$ by the above lemmas. Finally using Corollary~\ref{cor.AdjCov} we get an alternative characterization of the sheaf of $m$-adjoints, in fact, the usual definition when working in a category of desingularizable schemes ({\it e.g.}, for our case of characteristic zero).

\begin{cor}[Alternative Characterization of Adjoints]\label{cro.AltChar}
If $\pi: Y \to X$ is any desingularization then $\Omega_{X,adj}^m \cong \pi_*(\omega_Y^{\otimes m})$.
\end{cor}

\section{Computing Adjoints}\label{sec.Computing}

Now let $X \subset \P_\E^{l+1}$ be a projective hypersurface with defining homogeneous equation $F \in \E[x_0,\ldots,x_{l+1}]$ of degree $d$ (not equal to a coordinate hyperplane). For $0 \le i \le l+1$ we define the open sets $U_i \subset X$ obtained by intersection with the standard open covering sets $x_i \neq 0$ of $\P_{\E}^{l+1}$.

\begin{dfn}[Dualizing Sheaf]\label{dfn.DualSheaf}
 By $\omega_X^0 \subseteq \Omega_{X,rat}^1$ we denote the \emph{dualizing sheaf}. It is invertible and generated on $U_i$ by the form
\begin{gather*}
\gamma_i := \sigma_{i,j} \left(\frac{\partial F / \partial x_j}{x_i^{d-1}}\right)^{-1} \d \frac{x_0}{x_i} \wedge \dots \wedge \widehat{\d \frac{x_i}{x_i}} \wedge  \dots \wedge \widehat{\d \frac{x_j}{x_i}} \wedge \dots \wedge \d \frac{x_{l+1}}{x_i}
\end{gather*}
for any choice of $j \neq i$ where
\begin{gather*}
\sigma_{i,j} := \begin{cases}
(-1)^{i+j}   & \text{if $j<i$,} \\
(-1)^{i+j+1} & \text{if $j>i$.}
\end{cases}
\end{gather*}
(The hats here mean that the corresponding terms in the exterior product are to be excluded.)
\end{dfn}

Using the rules of calculus and the fact that $$0 = \d \frac{F}{x_i^d} = \sum_{0 \le k \le l+1, k \neq i} \frac{\partial F / \partial x_k}{x_i^{d-1}} \d \frac{x_k}{x_i}$$ holds in $\Omega_{X,rat}^1$ one proves that the definition is indeed independent of the choice of $j$. Because of local freeness we also have $(\omega_X^0)^{\otimes m} \subseteq \Omega_{X,rat}^m$ (meaning the natural map is an embedding).

\begin{lem}[Properties of the Dualizing Sheaf]\label{lem.DualSheaf}
For the dualizing sheaf we have:
\begin{itemize}
\item
$\omega_X^0 \cong \OO_X(d-l-2)$
\item
$(\omega_X^0)_p = (\Omega_{X,reg}^1)_p = (\Omega_{X,adj}^1)_p$ at all regular points $p$
\item
$\Omega_{X,adj}^m \subseteq (\omega_X^0)^{\otimes m}$ as subsheaves of $\Omega_{X,rat}^m$
\end{itemize}
\end{lem}
\begin{proof}
To prove the first statement one just shows, using the rules of calculus, that $\gamma_{i_1} = (x_{i_1}/x_{i_2})^{d-l-2} \gamma_{i_2}$. The same relation is fulfilled by the local generators $x_{i_1}^{d-l-2}$ and $x_{i_2}^{d-l-2}$.

We check the second statement for points $p \in U_i$. The forms $$\d \frac{x_0}{x_i} \wedge \dots \wedge \widehat{\d \frac{x_i}{x_i}} \wedge  \dots \wedge \widehat{\d \frac{x_j}{x_i}} \wedge \dots \wedge \d \frac{x_{l+1}}{x_i}$$ with $j \neq i$, which are all $\Gamma(U_i,\OO_{X})$-multiples of $\gamma_i$, clearly generate $\Gamma(U_i,\Omega_{X,reg}^1)$. Therefore we have $(\Omega_{X,reg}^1)_p \subseteq (\omega_X^0)_p$ at all points. Assuming moreover that $p$ is a regular point, there must be $j \neq i$ s.t.\ $(\partial F / \partial x_j)/x_i^{d-1} \not\in \mm_{X,p}$ and therefore $(\partial F / \partial x_j)/x_i^{d-1}$ is invertible in $\OO_{X,p}$. Choosing this $j$ in Definition~\ref{dfn.DualSheaf} one immediately sees that $\gamma_i$ is regular at $p$.

For the last statement we consider a generic projection $X \to Z$ to a hyperplane $Z \subseteq \P_{\E}^{l+1}$. We may assume that $Z$ is given by $x_0=0$ and that $F$ is monic in $x_0$.\todo{May we?} In this situation one can define a trace $\sigma_{X \mid Z}: \Omega_{X,rat}^m \to \Omega_{Z,rat}^m$ obtained from the trace of the field extension $\E(Z) \subseteq \E(X)$. By \cite[Satz 2.14]{EK:1978} and Lemma~\ref{lem.AdjReg2} we know that $(\omega_X^0)^{\otimes m} = \sigma_{X \mid Z}^{-1}(\Omega_{Z,reg}^m) = \sigma_{X \mid Z}^{-1}(\Omega_{Z,adj}^m)$. It remains to show that if $\alpha \in \Omega_{X,adj}^m$ then $\sigma_{X \mid Z}(\alpha) \in \Omega_{Z,adj}^m$.

Let $R \subset \E(Z)$ be a divisorial valuation ring. Further let $S_i \subset \E(X)$ (for $i$ in a finite index set) be the extensions, {\it i.e.}, divisorial valuation rings dominating $R$. Using Definition~\ref{dfn.AdjForm} (and the fact that completion is faithfully flat) we have to show that $\alpha \in \bigotimes (\bigwedge S_i \d S_i)$ for all $i$ implies $\sigma_{X \mid Z}(\alpha) \in \bigotimes (\bigwedge R \d R)$.\todo{Is for one $i$ sufficient?} By \cite[Prop.~VI.8.6.6]{NB:1989} the $S_i$ are localizations of the integral closure of $R$ in $\E(X)$. Then the statement follows from \cite[Satz 2.15]{EK:1978}.
\end{proof}

We want to see that, under certain additional assumptions, checking for adjointness involves only finitely many formal prime divisors.

\begin{lem}[Adjointness by Formal Desingularizations]\label{lem.FormalCrit}
Let $\SS$ be a formal desingularization of $X$ and $U \subset X$ an open subset. For $\eta \in \Omega_{X,rat}^m$ the following are equivalent:
\begin{itemize}
\item $\eta \in \Gamma(U,\Omega_{X,adj}^m)$
\item $\eta \in \Gamma(U \setminus \supp(\SS),\Omega_{X,adj}^m)$ and $\eta$ is regular at all $\varphi \in \SS$ with $\cent(\varphi) \in U$
\end{itemize}
\end{lem}
\begin{proof}
The first implication is trivial, so assume that the second condition is true. Let $\pi: Y \to X$ be a desingularization that is described by $\SS$ and set $V := \pi^{-1}(U)$.

By Corollary~\ref{cor.AdjCov} and Lemma~\ref{lem.AdjReg2} we have $\pi^\#(\Gamma(U,\Omega_{X,adj}^m)) = \Gamma(V,\Omega_{Y,adj}^m) = \Gamma(V,\Omega_{Y,reg}^m)$. Since $\pi$ induces an isomorphism $V \setminus \supp(\pi^*(\SS)) \cong U \setminus \supp(\SS)$ it remains to check that $\pi^\#(\eta)$ is regular in $\supp(\pi^*(\SS)) \cap V$. Since $Y$ is regular the locus of non-regularity of $\pi^\#(\eta)$ has pure codimension $1$. Hence, by Lemma~\ref{lem.RegRealized}, it is sufficient to check regularity of $\pi^\#(\eta)$ at the formal prime divisors in $\pi^*(\SS)$ with center in $V$. Equivalently, working on $X$, we have to check regularity of $\eta$ at the corresponding formal prime divisors in $\SS$. 
\end{proof}

In the following definition we assume that we have chosen free generators $\omega_{\varphi,m} \in \widetilde{\Omega}_{\F_\varphi\ll t\rr  \mid \E}^m$ ({\it e.g.}, $\omega_{\varphi,m} = (\d s_{\varphi,1} \wedge \dots \wedge \d s_{\varphi,l-1} \wedge \d t)^{\otimes m}$) for each formal prime divisor $\varphi: \spec \F_\varphi\ll t\rr  \to X$.

\begin{dfn}[Valuations Associated to Formal Prime Divisors]\label{dfn.AssocVals}
Let $\varphi: \spec \F_{\varphi}\ll t\rr \to X$ be a formal prime divisor. Define $\kappa_{\varphi} := \ord_t \circ \varphi^\# : \E(X) \to \Z$, which is a \emph{divisorial valuation}. We ``extend the valuation'' to $\Omega_{X,rat}^m$ as follows: If $\eta \in \Omega_{X,rat}^m$ and $\varphi^\#(\eta) = f \omega_{\varphi,m}$, then $\kappa_{\varphi}(\eta) := \kappa_{\varphi}(f)$. Finally we can also define ``valuations'' $\kappa_{\varphi} : \Gamma(X, \OO_X(k)) \to \Z$ for $k \in \Z$ by setting $\kappa_{\varphi}(f) := \kappa_{\varphi}(f/x_i^{k})$ for any index $i$ s.t.\ $\cent(\varphi) \in U_i$.
\end{dfn}

The map from $\Omega_{X,rat}^m$ is obviously well-defined because two free generators can differ only by a unit in $\F_\varphi\ll t\rr $ which has order $0$. We should also make sure that the definition for the map from $\Gamma(X, \OO_X(k))$ does not depend on the choice of the index $i$. Assume that $j \neq i$ is another index with $\cent(\varphi) \in U_j$. Then $f/x_i^{k} = (x_j/x_i)^{k} f/x_j^{k}$ and hence $\kappa_{\varphi}(f/x_i^{k}) = k\,\kappa_{\varphi}(x_j/x_i) + \kappa_{\varphi}(f/x_j^{k})$. Since $\cent(\varphi) \in U_i \cap U_j$ we have that $x_j/x_i \in \OO_{X,\cent(\varphi)}$ is invertible and so is $\varphi^\#(x_j/x_i) \in \F_\varphi\ll t\rr $. But then again $\kappa_{\varphi}(x_j/x_i) = \ord_t(\varphi^\#(x_j/x_i)) = 0$.

\begin{dfn}[Adjoint Order]\label{dfn.AdjOrder}
Let $\varphi: \spec \F_\varphi\ll t\rr  \to X$ be a formal prime divisor and $0 \le i \le l+1$ an index s.t.\ $\cent(\varphi) \in U_i$. We define the \emph{adjoint order} at $\varphi$ as $\alpha_{\varphi} := -\kappa_{\varphi}(\gamma_i)$.
\end{dfn}

This definition is again independent of the index $i$ by an analogous reasoning as above.

\begin{thm}[Global Sections of Twisted Pluricanonical Sheaves]\label{thm.FormalCrit}
Let $\SS$ be a formal desingularization of $X$. Then
\begin{gather*}
\Gamma(X, \OO_X(n) \otimes_{\OO_X} \Omega_{X,adj}^m) \cong \{ f \in \Gamma(X, \OO_X(n + m (d-l-2))) \mid \kappa_{\varphi}(f) \ge m\, \alpha_{\varphi} \text{ for all } \varphi \in \SS\}.
\end{gather*}
\end{thm}
\begin{proof}
By Lemma~\ref{lem.DualSheaf} we can view $\OO_X(n) \otimes_{\OO_X} \Omega_{X,adj}^m$ as a subsheaf of $\OO_X(n) \otimes_{\OO_X} (\omega_{X}^0)^{\otimes m} \cong \OO_X(n + m (d-l-2))$. Let $f \in \Gamma(X,\OO_X(n + m (d-l-2)))$ be a global section and $\eta$ its preimage in $\Gamma(X,\OO_X(n) \otimes_{\OO_X} (\omega_{X}^0)^{\otimes m})$ via this isomorphism. Projecting $\eta$ to the sections over $U_i$ we find $$\eta \mapsto x_i^n \otimes \frac{f}{x_i^{n+m(d-l-2)}}\gamma_i^{\otimes m}.$$ We have to check whether $f/x_i^{n+m(d-l-2)} \gamma_i^{\otimes m} \in \Gamma(U_i,\Omega_{X,adj}^m)$ for all $i$. Again by Lemma~\ref{lem.DualSheaf} this form is adjoint at all regular points. Applying now Lemma~\ref{lem.FormalCrit} it is equivalent to check that $$\kappa_{\varphi}(f/x_i^{n+m(d-l-2)} \gamma_i^{\otimes m}) \ge 0$$ for all $i$ and any formal prime divisor $\varphi \in \SS$ with $\cent(\varphi) \in U_i$. This again is equivalent to $$\kappa_{\varphi}(f/x_i^{n+m(d-l-2)}) \ge -m\kappa_{\varphi}(\gamma_i) = m\, \alpha_{\varphi}$$ for all $\varphi \in \SS$.
\end{proof}

\section{The Algorithm}\label{sec.Algorithm}

Concerning the actual computation we close with a few remarks and give an explicit algorithm. First since a hypersurface $X$ is in particular a complete intersection we know that $$\Gamma(X, \OO_X(n + m (d-l-2))) \cong (\E[x_0,\ldots,x_{l+1}] / \langle F \rangle)_{n + m (d-l-2)}$$ (see, for example, \cite[Exer.~III.5.5.(a)]{RH:1977}). If $\preccurlyeq$ is some well-ordering on exponents compatible with the group structure and $\mu_0$ is the leading exponent of $F$ w.r.t.\ $\preccurlyeq$ then we can write $$(\E[x_0,\ldots,x_{l+1}] / \langle F \rangle)_{n + m (d-l-2)} = \langle \underline{x}^{\mu} \mid \lvert \mu \rvert = n + m (d-l-2) \text{ and } \mu \preccurlyeq \mu_0 \rangle_{\E}.$$

Second we want to comment on the computation of adjoint orders (see Definition~\ref{dfn.AdjOrder}). Therefore we have to determine $\kappa_{\varphi}(\gamma_i)$. More generally let $\eta \in \Omega_{X,rat}^1$ be arbitrary. Further let $u_1,\ldots,u_{l} \in \E(X)$ be a transcendence basis over $\E$. As a generator of $\widetilde{\Omega}_{\F_\varphi\ll t\rr}^1$ we choose as before $\omega_{\varphi,1} = \d s_{\varphi,1} \wedge \dots \wedge \d s_{\varphi,l-1} \wedge \d t$ and we can write $\eta = f \d u_1 \wedge \dots \wedge \d u_{l}$. Then by the rules of calculus $$\varphi^\#(\eta) = \varphi^\#(f) \left\lvert \frac{\partial(\varphi^\#(u_1),\dots,\varphi^\#(u_{l}))}{\partial(s_{\varphi,1},\ldots,s_{\varphi,l-1},t)} \right\rvert \d s_{\varphi,1} \wedge \dots \wedge \d s_{\varphi,l-1} \wedge \d t$$ and hence $$\kappa_{\varphi}(\eta) = \kappa_{\varphi}(f) + \ord_t \left\lvert \frac{\partial(\varphi^\#(u_1),\dots,\varphi^\#(u_{l}))}{\partial(s_{\varphi,1},\ldots,s_{\varphi,l-1},t)} \right\rvert.$$ In order to compute the order of the Jacobian one can use approximative methods, {\it i.e.}, compute with truncations of the involved power series of sufficiently high precision. In any case one has to be able to compute in the universal module of differentials of the field extension $\F_\varphi \mid \E$. This proves still a little difficult in current computer algebra systems. It is therefore preferable to compute the adjoint orders, when possible, simultaneously with the formal desingularization. This involves essentially repeated application of the chain rule of differential calculus.

With these remarks and the above notation it is now obvious how to derive an algorithm. Correctness of the following is immediate by Theorem~\ref{thm.FormalCrit}, and termination is trivial because formal desingularizations can be computed and consist of \emph{finitely many} formal prime divisors.
\begin{algorithm}[H]
\caption{$Adjoints(F : \E[x_0, \dots, x_{l+1}], m: \N, n: \N) : 2^{\E[x_0, \dots, x_{l+1}]}$}\label{alg.Adjoints}
\begin{algorithmic}[1]
\REQUIRE an irreducible, homogeneous polynomial $F$ of degree $d$ (not equal to $x_i$ for any $0 \le i \le l+1$), defining $X \subset \P_{\E}^{l+1}$
\ENSURE a basis for the space of global sections of $\OO_X(n) \otimes_{\OO_X} \pi_*(\omega_Y^{\otimes m})$ represented by homogeneous polynomials of degree $n + m(d-l-2)$ where $\pi: Y \to X$ is any desingularization
\\[1ex]
\STATE Let $B \subset \E[x_0, \dots, x_{l+1}]$ be a set representing a basis of $(\E[x_0, \dots, x_{l+1}]/\langle F \rangle)_{n + m(d-l-2)}$;
\STATE Compute a formal desingularization $\SS$ of $X$ and adjoint orders $\alpha_\varphi$ for all $\varphi \in \SS$;
\STATE $C := 0 \in \E^{\infty \times \lvert B \rvert}$;
\COMMENT{a matrix with an undetermined number of rows}
\FOR{$\varphi \in \SS$}
\STATE $A := \sum_{b \in B} c_b \, Trunc(\varphi^\#(b), m\,\alpha_\varphi) = \sum_{0 \le j < m\,\alpha_\varphi} a_j t^j$;
\COMMENT{with $a_j$ linear in $\F_\varphi[c_b \mid b \in B]$}
\STATE $C := AddConstraints(C, \{a_j\}_{0 \le j < m\,\alpha_\varphi})$;
\ENDFOR
\STATE Let $K \subset \E^{\# B}$ be a basis of $\operatorname{ker}(C)$;
\STATE \textbf{return} $\{ \sum_{b \in B} c_b b \mid (c_b)_{b \in B} \in K \}$;
\end{algorithmic}
\end{algorithm}

Here $Trunc(\varphi^\#(b), m\,\alpha_\varphi)$ means the truncation of $\varphi^\#(b)$ at order $m\,\alpha_\varphi$. It remains to explain the function $AddConstraints$. It is meant to stack new rows on top of the matrix $C$, representing the linear constraints imposed by the formal prime divisor $\varphi$.

Assume $\E \subset \F' \subset \F$ is a tower of field extensions where $\F$ over $\F'$ is simple (algebraic or transcendental). Let $a := \sum_{b \in B} y_b c_b \in \F[c_b \mid b \in B]$. We want to find values $c_b \in \E$ s.t.\ the linear constraint $a=0$ is fulfilled. We are done if we know how to translate the constraint equivalently to a finite number of linear constraints over the smaller field $\F'$. Using this step recursively and considering the fact that $\F$ over $\E$ is finitely generated, we finally get a set of constraints with coefficients in $\E$. If $\sum_{b \in B} y_{b} c_b = 0$ is such a constraint, the function $AddConstraints$ would stack the row vector $(y_{b})_{b \in B}$ on top of the matrix $C$. We distinguish two cases:
\begin{itemize}
\item
If $\F$ over $\F'$ is algebraic, say, of degree $e+1$, choose a basis $\{f_r\}_{0 \le r \le e}$ of $\F$ as an $\F'$-vector space. Then $$0 = \sum_{b \in B} y_b c_b = \sum_{b \in B} c_b \sum_{0 \le r \le d} y_{b,r} f_r = \sum_{0 \le r \le e} \left(\sum_{b \in B} y_{b,r} c_b\right) f_r$$ holds if and only if $\sum_{b \in B} y_{b,r} c_b = 0$ for all $0 \le r \le e$.
\item
Now assume $\F = \F'(s)$ is transcendental. W.l.o.g. we may assume that the $y_b$ are actually polynomials in $\F'[s]$, otherwise multiply the equation by the common denominator. Let $e$ be the maximal degree of all the $y_b$. Then $$0 = \sum_{b \in B} y_b c_b = \sum_{b \in B} c_b \sum_{0 \le r \le e} y_{b,r} s^r = \sum_{0 \le r \le e} \left(\sum_{b \in B} y_{b,r} c_b\right) s^r$$ holds again if and only if $\sum_{b \in B} y_{b,r} c_b = 0$ for all $0 \le r \le e$.
\end{itemize}

\section{Example}

Let $\E := \Q$ and write $x := x_0, y := x_1, z := x_2, w := x_3$. The homogeneous polynomial $F := w^3y^2z+(xz+w^2)^3 \in \Q[x,y,z,w]$ of degree $d = 6$ defines a hypersurface $X \subset \P_{\Q}^3$, {\it i.e.}, $l = \dim(X) = 2$. We compute a formal desingularization $\SS$ using Algorithm~1 of \cite{TB:2007}. Amongst others, we get a formal prime divisor $\varphi: \spec \F_{\varphi}\ll t \rr \to X$ defined by the $\Q$-algebra homomorphism
\begin{gather*}
\varphi^\#: \Q[x,y,z,w]/\langle F \rangle \to \F_{\varphi}\ll t \rr:
\begin{cases}
x \mapsto 1, \\
y \mapsto -\tfrac{8}{s}t^3, \\
z \mapsto \tfrac{64}{s}t^6, \\
w \mapsto  - \tfrac{8}{s} \alpha t^3 - \tfrac{8}{s} t^4 + \tfrac{4}{s^2} \alpha t^5 + \tfrac{1}{s^3} \alpha t^7 + \tfrac{1}{2s^4} \alpha t^9 + \AO(t^{11})
\end{cases}
\end{gather*}
where $\F_{\varphi} = \Q(s)[\alpha]$ and $\alpha$ has minimal polynomial $\alpha^2 + s$.

First we want to compute the adjoint order of this formal prime divisor. Therefore we consider the rational differential form $$\frac{x^{d-1}}{\partial F / \partial w} \d \tfrac{y}{x} \wedge \d \tfrac{z}{x} = \frac{x^5}{6x^2z^2w + 12xzw^3 + 3y^2zw^2 + 6w^5} \d \tfrac{y}{x} \wedge \d \tfrac{z}{x}.$$ Now we apply the map induced by $\varphi^\#$ and find the differential form $$\frac{1}{\tfrac{786432}{s^4} \alpha t^{17} + \tfrac{1572864}{s^4} t^{18} - \tfrac{1966080}{s^5} \alpha t^{19} + \AO(t^{20})} \d -\tfrac{8}{s}t^3 \wedge \d \tfrac{64}{s}t^6.$$ According to Definition~\ref{dfn.AssocVals} we generate by $\d s \wedge \d t$ and rewrite this form again to $$\frac{1}{\tfrac{512}{s} \alpha t^{9} + \tfrac{1024}{s} t^{10} - \tfrac{1280}{s^2} \alpha t^{11} + \AO(t^{12})} \d s \wedge \d t.$$ The coefficient has order $-9$, so $\kappa_{\varphi} = 9$ as of Definition~\ref{dfn.AdjOrder}.

Assume now, we want to compute the global sections of $\pi_*(\omega_Y) \otimes_{\OO_X} \OO_X(1)$ (where $\pi: Y \to X$ is any desingularization, not necessarily the one described by $\SS$), {\it i.e.}, we have $m = 1$ and $n = 1$. We compute $n + m(d-l-2) = 1 + 1(6-2-2) = 3$. Therefore we first need a set $B$ projecting bijectively to the component of $\Q[x,y,z,w]/\langle F \rangle$ of homogeneous degree $3$. Since the defining equation is of degree $6$ we can choose the set of all monomials of degree $3$: $$B := \{x^3, x^2y, xy^2, y^3, x^2z, xyz, y^2z, xz^2, yz^2, z^3, x^2w, xyw, y^2w, xzw, yzw, z^2w, xw^2, yw^2, zw^2, w^3\}$$

Applying $\varphi^\#$ to the generic form of degree $3$ we find:
\begin{align*}
\varphi^\#(\textstyle \sum_{b \in B} c_b b) = & (c_{x^3})t^0 +
(- \tfrac{8}{s} c_{x^2y} - \tfrac{8}{s}\alpha c_{x^2w})t^3 +
(- \tfrac{8}{s} c_{x^2w})t^4 +
(\tfrac{4}{s^2}\alpha c_{x^2w})t^5 +\\
& (\tfrac{64}{s^2} c_{xy^2} + \tfrac{64}{s} c_{x^2z} + \tfrac{64}{s^2} \alpha c_{xyw} - \tfrac{64}{s} c_{xw^2})t^6 +
(\tfrac{1}{s^3} \alpha c_{x^2w} + \tfrac{64}{s^2} c_{xyw} + \tfrac{128}{s^2} \alpha c_{xw^2})t^7 +\\
& (- \tfrac{32}{s^3} \alpha c_{xyw} + \tfrac{128}{s^2} c_{xw^2})t^8 + \AO(t^9)
\end{align*}

A form is adjoint iff its $\varphi^\#$-image vanishes with order greater or equal to $m \kappa_{\varphi} = 1 \cdot 9 = 9$, {\it i.e.}, the coefficients of $t^0,\dots,t^8$ have to vanish. Viewing $B$ as an ordered basis we can write this as a matrix in $\Q(s)[\alpha]^{9 \times 20}$:
\begin{gather*}
\left(\begin{array}{rrrrrrrrrrrrrrrrrrrr}
1 & 0 & 0 & 0 & 0 & 0 & 0 & 0 & 0 & 0 & 0 & 0 & 0 & 0 & 0 & 0 & 0 & 0 & 0 & 0 \\
0 & 0 & 0 & 0 & 0 & 0 & 0 & 0 & 0 & 0 & 0 & 0 & 0 & 0 & 0 & 0 & 0 & 0 & 0 & 0 \\
0 & 0 & 0 & 0 & 0 & 0 & 0 & 0 & 0 & 0 & 0 & 0 & 0 & 0 & 0 & 0 & 0 & 0 & 0 & 0 \\
0 & -\tfrac{8}{s} & 0 & 0 & 0 & 0 & 0 & 0 & 0 & 0 & -\tfrac{8}{s}\alpha & 0 & 0 & 0 & 0 & 0 & 0 & 0 & 0 & 0 \\
0 & 0 & 0 & 0 & 0 & 0 & 0 & 0 & 0 & 0 & -\tfrac{8}{s} & 0 & 0 & 0 & 0 & 0 & 0 & 0 & 0 & 0 \\
0 & 0 & 0 & 0 & 0 & 0 & 0 & 0 & 0 & 0 & \tfrac{4}{s^2}\alpha & 0 & 0 & 0 & 0 & 0 & 0 & 0 & 0 & 0 \\
0 & 0 & \tfrac{64}{s^2} & 0 & \tfrac{64}{s} & 0 & 0 & 0 & 0 & 0 & 0 & \tfrac{64}{s^2}\alpha & 0 & 0 & 0 & 0 & -\tfrac{64}{s} & 0 & 0 & 0 \\
0 & 0 & 0 & 0 & 0 & 0 & 0 & 0 & 0 & 0 & \tfrac{1}{s^3}\alpha & \tfrac{64}{s^2} & 0 & 0 & 0 & 0 & \tfrac{128}{s^2}\alpha & 0 & 0 & 0 \\
0 & 0 & 0 & 0 & 0 & 0 & 0 & 0 & 0 & 0 & 0 & -\tfrac{32}{s^3}\alpha & 0 & 0 & 0 & 0 & \tfrac{128}{s^2} & 0 & 0 & 0
\end{array}\right)
\end{gather*}

The third row from the bottom corresponds to the constraint $$\tfrac{64}{s^2} c_{xy^2} + \tfrac{64}{s} c_{x^2z} + \tfrac{64}{s^2} \alpha c_{xyw} - \tfrac{64}{s} c_{xw^2} = 0.$$ Reordering this in terms of the basis $\alpha^0, \alpha^1$ of $\Q(s)[\alpha]$ over $\Q(s)$ we find $$(\tfrac{64}{s^2} c_{xy^2} + \tfrac{64}{s} c_{x^2z} - \tfrac{64}{s} c_{xw^2})\alpha^0 + (\tfrac{64}{s^2} c_{xyw})\alpha^1 = 0$$ and hence the two equivalent constraints $$\tfrac{64}{s^2} c_{xy^2} + \tfrac{64}{s} c_{x^2z} - \tfrac{64}{s} c_{xw^2} = 0 \text{ and } \tfrac{64}{s^2} c_{xyw} = 0.$$

The second one is clearly equivalent to $64 c_{xyw} = 0$ which is already a constraint over $\Q$. The first one can be multiplied by $s^2$ and rewritten in terms of $s^0$ and $s^1$: $$(64 c_{xy^2})s^0 + (64 c_{x^2z} - 64 c_{xw^2})s^1 = 0$$ This yields another two constraints $$64 c_{xy^2} = 0 \text{ and } 64 c_{x^2z} - 64 c_{xw^2} = 0$$ over $\Q$. Altogether the third row from the bottom corresponds to the following matrix in $\Q^{3 \times 20}$:
\begin{gather*}
\left(\begin{array}{rrrrrrrrrrrrrrrrrrrr}
0 & 0 & 64 & 0 & 0 & 0 & 0 & 0 & 0 & 0 & 0 & 0 & 0 & 0 & 0 & 0 & 0 & 0 & 0 & 0 \\
0 & 0 & 0 & 0 & 64 & 0 & 0 & 0 & 0 & 0 & 0 & 0 & 0 & 0 & 0 & 0 & -64 & 0 & 0 & 0 \\
0 & 0 & 0 & 0 & 0 & 0 & 0 & 0 & 0 & 0 & 0 & 64 & 0 & 0 & 0 & 0 & 0 & 0 & 0 & 0
\end{array}\right)
\end{gather*}

Treating all rows similarly, stacking the computed matrices on top of each other and skipping zero rows we obtain:
\begin{gather*}
\left(\begin{array}{rrrrrrrrrrrrrrrrrrrr}
1 & 0 & 0 & 0 & 0 & 0 & 0 & 0 & 0 & 0 & 0 & 0 & 0 & 0 & 0 & 0 & 0 & 0 & 0 & 0 \\
0 & -8 & 0 & 0 & 0 & 0 & 0 & 0 & 0 & 0 & 0 & 0 & 0 & 0 & 0 & 0 & 0 & 0 & 0 & 0 \\
0 & 0 & 0 & 0 & 0 & 0 & 0 & 0 & 0 & 0 & -8 & 0 & 0 & 0 & 0 & 0 & 0 & 0 & 0 & 0 \\
0 & 0 & 0 & 0 & 0 & 0 & 0 & 0 & 0 & 0 & -8 & 0 & 0 & 0 & 0 & 0 & 0 & 0 & 0 & 0 \\
0 & 0 & 0 & 0 & 0 & 0 & 0 & 0 & 0 & 0 & 4 & 0 & 0 & 0 & 0 & 0 & 0 & 0 & 0 & 0 \\
0 & 0 & 64 & 0 & 0 & 0 & 0 & 0 & 0 & 0 & 0 & 0 & 0 & 0 & 0 & 0 & 0 & 0 & 0 & 0 \\
0 & 0 & 0 & 0 & 64 & 0 & 0 & 0 & 0 & 0 & 0 & 0 & 0 & 0 & 0 & 0 & -64 & 0 & 0 & 0 \\
0 & 0 & 0 & 0 & 0 & 0 & 0 & 0 & 0 & 0 & 0 & 64 & 0 & 0 & 0 & 0 & 0 & 0 & 0 & 0 \\
0 & 0 & 0 & 0 & 0 & 0 & 0 & 0 & 0 & 0 & 0 & 64 & 0 & 0 & 0 & 0 & 0 & 0 & 0 & 0 \\
0 & 0 & 0 & 0 & 0 & 0 & 0 & 0 & 0 & 0 & 1 & 0 & 0 & 0 & 0 & 0 & 0 & 0 & 0 & 0 \\
0 & 0 & 0 & 0 & 0 & 0 & 0 & 0 & 0 & 0 & 0 & 0 & 0 & 0 & 0 & 0 & 128 & 0 & 0 & 0 \\
0 & 0 & 0 & 0 & 0 & 0 & 0 & 0 & 0 & 0 & 0 & 0 & 0 & 0 & 0 & 0 & 128 & 0 & 0 & 0 \\
0 & 0 & 0 & 0 & 0 & 0 & 0 & 0 & 0 & 0 & 0 & -32 & 0 & 0 & 0 & 0 & 0 & 0 & 0 & 0
\end{array}\right)
\end{gather*}

Doing the same computation for all the other formal prime divisors in $\SS$ one computes a huge matrix which turns out to have a one-dimensional kernel spanned by
\begin{gather*}
\left(\begin{array}{rrrrrrrrrrrrrrrrrrrr}
0 & 0 & 0 & 0 & 0 & 0 & 0 & 0 & 0 & 0 & 0 & 0 & 0 & 1 & 0 & 0 & 0 & 0 & 0 & 1 \\
\end{array}\right).
\end{gather*}
The entries of this vector are the coefficients of the form $xzw + w^3$. So we have $$\Gamma(X, \pi_*(\omega_Y) \otimes_{\OO_X} \OO_X(1)) \cong \langle xzw + w^3 \rangle_{\Q}.$$


\bibliographystyle{amsplain}
\bibliography{bibfile}

\end{document}